\documentclass[10pt,a4paper]{article}
\usepackage{fullpage}
\usepackage{amsfonts,amsmath,amssymb}
\usepackage{amsthm}
\usepackage{graphicx}\usepackage{sectsty}
\usepackage{authblk}

\theoremstyle{plain}
\newtheorem{theorem}{Theorem}[section]
\newtheorem{lemma}[theorem]{Lemma}
\newtheorem{proposition}[theorem]{Proposition}
\newtheorem{corollary}[theorem]{Corollary}

\theoremstyle{definition}

\theoremstyle{remark}

\numberwithin{equation}{section}

\sectionfont{\large}

\begin{document}
\title{Stability of Inverse Resonance Problem on the Half Line }
\author{Lung-Hui Chen$^1$}\maketitle\footnotetext[1]{General Education Center, Ming Chi University of Technology, New Taipei City, 24301, Taiwan. Email:
mr.lunghuichen@gmail.com.}\maketitle
\begin{abstract}
We consider the inverse resonance problem in one-dimensional scattering theory. The scattering matrix consists of $2\times 2$ entries of meromorphic functions, which are quotients of certain Fourier transform. The resonances are expressed as the zeros of Fourier transform of wave field. For compactly-supported perturbation, we are able to quantitatively estimate the zeros and poles of each meromorphic entry. The size of potential support is connected to the zero distribution of scattered wave field. We derive the inverse stability on scattering source based on certain knowledge on the perturbation theory of resonances. When the resonances are distributed regularly, there is certain natural stability through the value distribution theory.
\end{abstract}
\section{Introduction}
In this note, we consider the scattering theory of one-dimensional Schr\"{o}dinger equation
\begin{equation}\label{1.1}
\psi''(k,x)+k^{2}\psi(k,x)=V(x)\psi(k,x),\,x\in\mathbb{R},
\end{equation}
satisfying following assumptions:
\begin{eqnarray}\label{1.2}
&&V(x)\in \mathcal{C}^{2}((0,\infty);\mathbb{R});\\
&&V'(0)=0,\,V(0)=1;\\\label{1.4}
&&\mbox{For every }N, \mbox{there exists a constant }C_{N}\mbox{ such that }V(x)\leq C_{N}e^{-N|x|}.
\end{eqnarray}
For example, a Gaussian potential satisfies the~(\ref{1.2}--\ref{1.4}). 
It is known that the Jost solutions of~(\ref{1.1}) from the left and right satisfy
\begin{eqnarray}\label{AB}
&&\phi_{+}(x,k)=\left\{%
\begin{array}{ll}
e^{ikx},&x\rightarrow\infty;\vspace{5pt}\\
\frac{\hat{X}(k)}{ik}e^{ikx}+\frac{\hat{Y}(-k)}{ik}e^{-ikx},&x\rightarrow-\infty;\label{CD}
\end{array}%
\right.\\
&&\phi_{-}(x,k)=\left\{%
\begin{array}{ll}
\frac{\hat{X}(k)}{ik}e^{-ikx}+\frac{\hat{Y}(k)}{ik}e^{ikx},&x\rightarrow\infty;\vspace{5pt}\\
e^{-ikx},&x\rightarrow-\infty,
\end{array}%
\right.
\end{eqnarray}
where the Fourier transforms $\hat{X}(k)$ and $\hat{Y}(k)$ are entire functions in $\mathbb{C}$. We define the scattering matrix $S(x)$ as
\begin{equation}\label{S}
S(k)=\left[\begin{array}{cc}\frac{ik}{\hat{X}(k)} & \frac{\hat{Y}(k)}{\hat{X}(k)} \vspace{7pt}\\\frac{\hat{Y}(-k)}{\hat{X}(k)} & \frac{ik}{\hat{X}(k)}\end{array}\right]:=\left[\begin{array}{cc}T(k) & R(k) \vspace{7pt}\\L(k) & T(k)\end{array}\right],
\end{equation}
in which
\begin{equation}\label{7}
T(k):=\frac{ik}{\hat{X}(k)}
\end{equation}
is the transmission coefficient, and $R(k)$ and $L(k)$ are the reflection coefficients from the right and left respectively. 
In particular, the zeros of $\hat{X}$, say, $\{\sigma_{j}\}$ are the {\bf resonances} of~(\ref{1.1}), satisfying $|\sigma_{1}|\leq|\sigma_{2}|\leq|\sigma_{3}|\leq\cdots$. Physically, $L^{2}$-eigenvalues are states in which the particles are permanently localized. On the contrast, resonances correspond to quasi-stationary (metastable) states that exist only for a finite time. Resonances that are close to the real axis appear as bumps in the scattering cross section and  measured feasibly in the laboratory. We refer more details to \cite{Dyatlov,Marletta,Mellin}, and we denote the resonant sequence as
$$\Sigma:=\{z_{n}\}.$$ 
The scattering matrix $S(k)$ is meromorphic in $\mathbb{C}$, and its poles in $\{\Im k>0\}$ are the square roots of $L^{2}$-eigenvalues of~(\ref{1.1}). We understand the eigenvalues and resonances are directly observable in spectrometers. \par
To locate the resonant sequence of~(\ref{1.1}), we use Froese's theorem \cite[Theorem\,1.3]{Froese} that if $V$ is a super-exponentially decreasing potential satisfying his Hypothesis 5.1 \cite[p.\,261]{Froese}, then the asymptotic distribution of resonances is identical to the asymptotic distribution of zeros for the function $\hat{V}(2k)\hat{V}(-2k)$. The~(1.2) is trivially a super-exponentially decreasing potential. In \cite[p.\,257]{Froese}, we have the scattering matrix of the form:
$$\left[\begin{array}{cc}1 & 0 \vspace{7pt}\\0 & 1\end{array}\right]+\left[\begin{array}{cc}T_{11} & T_{12}\vspace{5pt}\\T_{21} & T_{22}\end{array}\right].$$
\par
 
In inverse resonance problem, we consider to determine the potential $V$ from the resonances of~(\ref{1.1}) which includes the $L^{2}$-eigenvalues. Such an inverse process to identify the knowledge of  the emitting source by measuring all sorts of respects of the emittance or perturbed wave field in observational area has been research issues ever since the days of A. Sommerfeld and E. Rutherford. The inverse resonance problem of Schr\"{o}dinger operator on the half line has been studied in \cite{Dyatlov,Korotyaev2,Marletta}. Previously, Hitrik \cite{Hitrik}, Marletta \cite{Marletta}, Korotyaev \cite{Korotyaev2} and Bledsoe \cite{Beldsoe} studied the stability of inverse resonance problem. 
We state the following uniform stability result for $\mathcal{C}^{2}$-perturbation in this paper.
\begin{theorem}\label{11}
Let assume the Hypothesis 5.1 in \cite[p.\,261]{Froese},
and $\Sigma^{j}:=\{z_{n}^{j}\}_{n=1}^{\infty}$ be the resonant sequence of $V^{j}\in\mathcal{C}^{2}((0,\infty))$, $j=1,2$. Then, for any $\epsilon>0$,  there exists $\delta>0$ such that $|V^{1}(x)/V^{2}(x)-1|<\epsilon$, whenever $$0\leq\sup_{j}|\sigma_{j}^{1}-\sigma_{j}^{2}|<\delta.$$
\end{theorem}
There is a straightforward classic result from complex analysis. Let's say
\begin{eqnarray}
&&\hat{V}(z)=\hat{V}^{1}(0)e^{iCz}\prod_{\{z_{n}:{\rm all\, resonances}\}}(1-\frac{z}{z_{n}});\\
&&\hat{V}_{R}(z)=\hat{V}(0)e^{iCz}\prod_{|z_{n}|<R}(1-\frac{z}{z_{n}}).
\end{eqnarray}
Then, $$\lim_{R\rightarrow\infty}\hat{V}_{R}(z)=\hat{V}(z),$$ pointwise under certain regular assumption,
which is the natural stability with respect to the quantity of the resonances due to Cartwright's theory \cite[p.\,251]{Levin}. This meets the physicist's intuition that whenever one collects enough resonances, the laboratory measurement is close to the actual scattering source.

\section{Lemmata}
\begin{lemma}
For real-value potential $V$, $\overline{\widehat{V}(-2k)}=\widehat{V}(2k)$ for real $k$.
\end{lemma}
\begin{proof}
Note that $\widehat{V}(2k)=\int_{0}^{1}V(x)e^{{-2ikx}}dx$. Thus, $$\widehat{V}(2k)=\int_{0}^{1}V(x)e^{{-2ikx}}dx=\overline{\int_{0}^{1}V(x)e^{{2ikx}}dx}=\overline{\widehat{V}(-2k)},$$for real $k$. This proves the lemma.
\end{proof}
\begin{lemma}\label{Froese}
The zeros of $\widehat{V}(2k)\widehat{V}(-2k)$ are the resonances of~(\ref{1.1}) asymptotically. 
\end{lemma}
\begin{proof}
This is Froese's theorem \cite[Theorem\,1.3]{Froese}.
\end{proof}
\begin{theorem}[Cartwright]\label{Cartwright}
Let $f$ be an entire function of exponential type with zero set $\{a_{k}\}$. We assume $f$ satisfies one of the
following conditions:
\begin{equation}\nonumber
\mbox{ the integral
}\int_{-\infty}^\infty\frac{\ln^+|f(x)|}{1+x^2}dx\mbox{ exists}.
\end{equation}
\begin{equation}\nonumber
|f(x)|\mbox{ is bounded on the real axis}.
\end{equation}
Then,
\begin{enumerate}
\item all of the zeros of the function $f(z)$, except possibly
those of a set of zero density, lie inside arbitrarily small
angles $|\arg z|<\epsilon$ and $|\arg z-\pi|<\epsilon$, where the
density
\begin{equation}\label{3.9}
\Delta_f(-\epsilon,\epsilon)=\Delta_f(\pi-\epsilon,\pi+\epsilon)=\lim_{r\rightarrow\infty}
\frac{N(f,-\epsilon,\epsilon,r)}{r}
=\lim_{r\rightarrow\infty}\frac{N(f,\pi-\epsilon,\pi+\epsilon,r)}{r},
\end{equation}
is equal to $\frac{d}{2\pi}$, where $d$ is the width of the
indicator diagram \cite{Boas} and \cite[p.\,251]{Levin}. Furthermore, the limit
$\delta=\lim_{r\rightarrow\infty}\delta(r)$ exists, where
$$
\delta(r):=\sum_{\{|a_k|<r\}}\frac{1}{a_k};
$$
\item moreover,
\begin{equation}\nonumber
\Delta_f(\epsilon,\pi-\epsilon)=\Delta_f(\pi+\epsilon,-\epsilon)=0;
\end{equation}
\item the function $f(z)$ can be represented in the form
\begin{equation}\nonumber
f(z)=cz^me^{i\kappa
z}\lim_{r\rightarrow\infty}\prod_{\{|a_k|<r\}}(1-\frac{z}{a_k}),
\end{equation}
where $c,m,\kappa$ are constants and $\kappa$ is real;
\item the indicator
function of $f$ is of the form
\begin{equation}
h_f(\theta)=\sigma|\sin\theta|.
\end{equation}
\end{enumerate}
\end{theorem}
\begin{proof}
We refer the statements, notations, and details to Levin \cite[p.\,251]{Levin}.
\end{proof}
\begin{corollary}
The Fourier transform $\widehat{V}(k)$ is of order one and of completely regular growth.
\end{corollary}
Now we consider the following expansion method of Fourier transform.
\begin{lemma}\label{2100}
If $\phi(t)$ is $\mathcal{C}^{N}[\alpha,\beta]$, then
$$\int_{\alpha}^{\beta}\phi(t)e^{ixt}dt=B_{N}(x)-A_{N}(x)+R_{N}(x),$$
where 
\begin{eqnarray*}
&&A_{N}(x)=\sum_{n=0}^{N-1}i^{n-1}\phi^{(n)}(\alpha)x^{{-n-1}}e^{ix\alpha};\nonumber\\
&&B_{N}(x)=\sum_{n=0}^{N-1}i^{n-1}\phi^{(n)}(\beta)x^{{-n-1}}e^{ix\beta};\nonumber\\
&&R_{N}(x)=(-ix)^{{-N}}\int_{\alpha}^{\beta}e^{ixt}\phi^{(N)}(t)dt=o(\frac{1}{|x|^{N}}),\mbox{ as }x\rightarrow\infty.
\end{eqnarray*}
\end{lemma}
\begin{proof}
The proof is done by integration by parts, and we refer the details to Erdelyi \cite[p.\,47]{Erdelyi}.
\end{proof}
\begin{lemma}
The following asymptotic holds for real $k$:
\begin{eqnarray}\label{VV}
4z^{2}\hat{V}(2z)\hat{V}(-2z)=1+o(\frac{1}{z}),
\end{eqnarray}
in which all little o-terms are analytic.
\end{lemma}
\begin{proof}
Using previous lemma for $N=2$, we obtain $$\int_{0}^{\infty}e^{-izx}V(x)dx=\frac{1}{iz}+o(\frac{1}{z^{2}}).$$
Hence, $$\hat{V}(z)\hat{V}(-z)=\frac{1}{z^{2}}+o(\frac{1}{z^{3}}).$$
This proves the lemma.
\end{proof}
Most importantly, we now apply the Wilder's type of theorems for
exponential polynomials
\cite{Dickson,Dickson2,Levin}. 
We consider the entire functions in the form
\begin{equation}\label{sum}
f(z)=\sum_{j=1}^nA_jz^{m_j}[1+\epsilon(z)]e^{\omega_jz},\,z\in\mathbb{C},
\end{equation}
where $n>1$, $A_j$, and $\omega_j$ are complex numbers such that
$A_j\neq0$ and the $\omega_j$ are distinct; the $m_j$ are
non-negative integers; the functions $\epsilon$ are analytic for
$|z|\geq r_0\geq0$ with $\lim_{z\rightarrow\infty}\epsilon(z)=0$.
When we are talking about the zeros of $f(z)$, we are referring to
its zeros outside certain open ball around the origin.
\par
First of all, we set up
the following quantities to the $f(z)$ in~(\ref{sum}): Let $Q$ be
the \textbf{broken line} given by the $\overline{\omega}_j$ given
in~(\ref{sum}) with
$\overline{\omega}_1,\ldots,\overline{\omega}_\sigma$ as its
vertices. We label the indices counterclockwise. Let $L_k$ be
the line segment $[\overline{\omega}_k,\overline{\omega}_{k+1}]$
and
$$\phi_k:=\arg\{\overline{\omega}_k-\overline{\omega}_{k+1}\},$$ in
$[-\frac{\pi}{2},\frac{3\pi}{2})$. Let
\begin{equation}\label{EK}
e_k=e^{i\phi_k}.
\end{equation}
Certain $\overline{\omega}_p$ on $L_k$ are assigned doubly indexed
subscripts as follows: let the convex hull of
$\overline{\omega}_{k},\,\overline{\omega}_{k+1}$ and
\begin{equation}\label{tau}
\tau_p=\overline{\omega}_p+im_pe_k,
\end{equation}
in which
$\overline{\omega}_p$ on $L_k$; assign subscripts
$j=1,\cdots,\sigma_k$ to $\omega_{kj}$ so that
$\omega_{k1}=\omega_k,\,\omega_{k\sigma_k}=\omega_{k+1}$ and
$\tau_{kj}$ are vertices of this convex hull and preceding in a
counterclockwise direction from $\overline{\omega}_k+im_ke_k$ to
$\overline{\omega}_{k+1}+im_{k+1}e_k$. Whenever the $\tau_p$ is doubly-indexed, the $m_p$ has to be indexed instantaneously. For
$j=1,\cdots,\sigma_{k}-1$,
$$L_{kj}:=[\tau_{kj},\tau_{kj+1}];$$
\begin{equation} \label{mu}
\mu_{kj}:=\frac{m_{kj}-m_{kj+1}}{(\omega_{kj}-\omega_{kj+1})e_k},
\end{equation}
which is real and $n_{kj}$ is the number of $\tau_p$ on $L_{kj}$.
We refer to a comprehensive graph for the zeros of~(\ref{sum}) to \cite[p.\,10]{Dickson2}.
\par
Moreover, for  $j=1,\cdots,\sigma_{k}-1$ and $H>0$, we define
\begin{equation}\label{V}
V_{kj}(H):=\{z|\,\Im(z/e_k)\geq0,\,|\Re(z/e_k)+\mu_{kj}\log|z||\leq
H\};
\end{equation}
$T_k(\theta)$ is defined to be a closed sector with vertex at zero
of opening $2\theta$ about the outward normal to $L_k$ through the
origin. For the same $k$ and $j$ and each triple of reals
$(\alpha,s,H)$, $s>0$ and $H>0$, the set
\begin{equation}\nonumber
R_{kj}(\alpha,s,H):=\{z|\,\Im(z/e_k)+\mu_{kj}\arg
z\in[\alpha,\alpha+s],\,|\Re(z/e_k)+\mu_{kj}\log|z||\leq H\},
\end{equation}
where $\arg z\in(\phi_k,\phi_k+\pi)$. It's a curvilinear tubular
neighborhood along some direction, and we can infer that $R_{kj}(\alpha,s,h)$ is in $V_{kj}(h)\cap T_k(\theta)$.
Let us collect a few facts from \cite{Dickson,Dickson2}.
\begin{lemma}\label{214}
The following properties hold.
\begin{enumerate}
\item The size of $R_{kj}(\alpha,s,H)$is approximately a rectangle
of dimension $s$ by $2H$ for large $\alpha$. \item $V_{kj}$ is
individually connected and disjoint to each other for each pair of
$(k,j)$ when $|z|$ is large. \item The boundaries of $V_{kj}$ are
logarithmic to the outward normal to $L_k$ to the exterior of $Q$.
\item For each fixed $\theta,\,H$, the subsets of $V_{kj}(H)$ are
in $T_k(\theta)$ for large $z$. \item For fixed $\theta$ and $H$,
\begin{equation}
 R_{kj}(\alpha,s,H)\subset V_{kj}\cap T_k(\theta),\,\forall
 s>0,\mbox{ for large }\alpha.
\end{equation}
\end{enumerate}
\end{lemma}
Let us define
\begin{equation}\nonumber
N_f(R_{kj}(\alpha,s,H)):=\{\mbox{ the numbers of zeros of }f(z)\mbox{ in
}R_{kj}(\alpha,s,H)\},
\end{equation}
where the zeros are counted according to their multiplicity, and we have the
following theorem.
\begin{theorem}[Dickson]\label{Dickson}
Let $f(z)$ be given as in~(\ref{sum}). Then, there exists $H>0$
such that \begin{enumerate}
               \item all but a finite number of zeros of $f$ of modulus
greater than $r_0$ are in $\cup_{k,j}V_{kj}$;
               \item
for each pair
of positive reals $\epsilon$ and $s_0$, there exists an
$\alpha_0=\alpha_0(\epsilon,s_0)$ such that whenever
$\alpha\geq\alpha_0$ and $s\geq s_0$,
\begin{equation}\label{2.9}
\big|N_f(R_{kj}(\alpha,s,H))-s|\omega_{kj+1}-\omega_{kj}|/(2\pi)\big|<n_{kj}-1+\epsilon,
\end{equation}
in which $n_{kj}$ is given in~(\ref{mu}).
\end{enumerate}
\end{theorem}
\begin{proof}
We refer the proof to Dickson \cite{Dickson}.
\end{proof}
\begin{proposition}\label{D}
Let $\{z_{n}^{j}\}_{n=1}^{\infty}$ be the point sequence of zeros of $\hat{V}^{j'}(2z)\hat{V}^{j'}(-2z)$ such that $|z_{1}^{j'}|<|z_{2}^{j'}|<|z_{3}^{j'}|<\cdots$, $j'=1,2.$ For $0\leq\delta\ll 1$ and $0\leq\sup_{n}|z^{1}_{n}-z^{2}_{n}|<\delta$, the difference of numbers of $\{z^{1}_{n}\}$ and $\{z^{2}_{n}\}$ inside  \begin{eqnarray*}
S_{R}:=\{z\in\mathbb{C}|\,0\leq\Re z\leq R,\,-K\leq\Im z\leq  K\},
\end{eqnarray*} is uniformly bounded.
\end{proposition}
\begin{proof}
We apply Theorem \ref{Dickson}. For our exponential polynomial 
$16z^{4}\hat{V}^{j'}(2z/i)\hat{V}^{j'}(-2z/i)$ in~(\ref{VV}), we deduce
\begin{eqnarray*}
&&k=1,\\
&&e_{1}=0,\\
&&\omega_{11}=-2,\,\omega_{12}=0,\,\omega_{13}=2;\\
&&m_{12}=m_{23}=0;\\
&&\mu_{12}=\mu_{23}=0;\\
&&n_{11}=n_{12}=2.
\end{eqnarray*}
Therefore, we use~(\ref{2.9}) to obtain
\begin{equation}\nonumber
\big|N_{\hat{V}^{^{j'}}(2z)\hat{V}^{^{j'}}(-2z)}(R_{1j}(\alpha,s,K))-s|\omega_{1j+1}-\omega_{1j}|/(2\pi)\big|<1+\epsilon.
\end{equation} 
Let us choose $\alpha\gg1$ large and $s=2\pi/|\omega_{1j+1}-\omega_{1j}|$ so that
\begin{equation}\label{2.11}
\big|N_{\hat{V}^{^{j'}}(2z)\hat{V}^{^{j'}}(-2z)}(R_{1j}(\alpha,2\pi/|\omega_{1j+1}-\omega_{1j}|,K))-1\big|<1+\epsilon,\,j'=1,2,\,j=1,2.
\end{equation} 
Hence, for small $0\leq\delta\ll1$, we have $z^{1}_{n}$ and $z^{2}_{n}$ falling in the same rectangle for large $n$.
This proves the proposition.
\end{proof}
\begin{lemma}\label{L29}
For $z^{4}\hat{V}(z)\hat{V}(-z)=\frac{1}{2}(1+o(1))+(1+o(1))e^{-iz}+(1+o(1))e^{iz},$ we conclude that the function $z^{4}\hat{V}(z)\hat{V}(-z)$ has only real zeros.
\end{lemma}
\begin{proof}
We verify the conditions provided in Duffin and Schaeffer \cite[p.\,236]{Duffin}.
Let us note that 
\begin{eqnarray*}
z^{4}\hat{V}(z)\hat{V}(-z)&=&\frac{1}{2}(1+o(1))+(1+o(1))e^{-iz}+(1+o(1))e^{iz}\\&=&2\cos z+o(1)e^{-iz}+o(1)e^{-iz}+\frac{1}{2}(1+o(1)),
\end{eqnarray*}
in which, for real $z$,
\begin{eqnarray*}
|o(1)e^{-iz}+o(1)e^{-iz}+\frac{1}{2}(1+o(1))|\leq \frac{1}{2}+o(1).
\end{eqnarray*}
We also note that $\frac{z^{4}\hat{V}(z)\hat{V}(-z)}{2}$ is real on the real axis. Therefore, we deduce from \cite{Duffin} that $z^{4}\hat{V}(z)\hat{V}(-z)$ has only real zeros.
\end{proof}
\section{Proof of Theorem \ref{11}}
\begin{proof}
We begin with the Lemma \ref{Froese} and assumption that the zero set of entire functions $\hat{V}^{1}(2z)\overline{\hat{V}^{1}}(2z)$ and $\hat{V}^{2}(2z)\overline{\hat{V}^{2}}(2z)$ are close to each other in $\mathbb{C}$, that is,
$$0\leq\sup_{j}|\sigma_{j}^{1}-\sigma_{j}^{2}|<\delta.$$ 
Let us set
\begin{eqnarray*}
&&S_{R}=\{z\in\mathbb{C}|\,0\leq\Re z\leq R,\,-K\leq\Im z\leq K\};\\
&&N^{j}(R):=\mbox{the number of zeros }\{z_{n}^{j}\}\mbox{ of }\hat{V}^{j}(2z)\overline{\hat{V}^{j}}(2z)\mbox{ inside }S_{R}.
\end{eqnarray*}
Using Theorem \ref{Cartwright}, we have
\begin{eqnarray}
&&\hat{V}^{1}(2z)\overline{\hat{V}^{1}}(2z)=c^{1}z^{m^{1}}e^{2iz}\lim_{R\rightarrow \infty}\prod_{z^{1}_{n}\in S_{R}}(1-\frac{z}{z_{n}^{1}}),\,z=x+iy;\\
&&\hat{V}^{2}(2z)\overline{\hat{V}^{2}}(2z)=c^{2}z^{m^{2}}e^{2iz}\lim_{R\rightarrow \infty}\prod_{z_{n}\in S_{R}}(1-\frac{z}{z_{n}^{2}}),\,z=x+iy,
\end{eqnarray}
where the exponential factor $2iz$ is related to the width of indicator diagram of $\hat{V}^{j}(2z)\overline{\hat{V}^{j}}(2z)$ by Cartwright theory, and we refer the details to \cite[p.\,251]{Levin}.
We have assumed that zero is not a bound state, so we deduce
\begin{eqnarray}
&&\hat{V}^{1}(2z)\overline{\hat{V}^{1}}(2z)=\hat{V}^{1}(0)\overline{\hat{V}^{1}}(0)e^{2iz}\lim_{R\rightarrow \infty}\prod_{z^{1}_{n}\in S_{R}}(1-\frac{z}{z_{n}^{1}});\\
&&\hat{V}^{2}(2z)\overline{\hat{V}^{2}}(2z)=\hat{V}^{2}(0)\overline{\hat{V}^{2}}(0)e^{2iz}\lim_{R\rightarrow \infty}\prod_{z^{1}_{n}\in S_{R}}(1-\frac{z}{z_{n}^{2}}).
\end{eqnarray}
Using assumption~(\ref{1.4}), we obtain
\begin{eqnarray}
&&\hat{V}^{1}(2z)\overline{\hat{V}^{1}}(2z)=e^{2iz}\lim_{R\rightarrow  \infty}\prod_{z^{1}_{n}\in S_{R}}(1-\frac{z}{z_{n}^{1}});\\
&&\hat{V}^{2}(2z)\overline{\hat{V}^{2}}(2z)=e^{2iz}\lim_{R\rightarrow\infty}\prod_{z^{1}_{n}\in S_{R}}(1-\frac{z}{z_{n}^{2}}).
\end{eqnarray}
Let us set
\begin{eqnarray}
&&\Pi^{1}_{R}(z):=\prod_{z^{1}_{n}\in S_{R}}(1-\frac{z}{z_{n}^{1}});\\
&&\Pi^{2}_{R}(z):=\prod_{z^{1}_{n}\in S_{R}}(1-\frac{z}{z_{n}^{2}}),\,R\gg0,
\end{eqnarray}
and consider the contour integral of $\Pi^{j}_{R}(z)$ counter-clockwise around strip $S_{R}$.
Using Proposition \ref{D} and the assumption of Theorem \ref{11}, there is no zero on the contour $\partial S_{R}$ for suitable $R\gg1$ and $K>0$.
Now we apply the argumentation principle and Cauchy-Riemann conditions to deduce \begin{eqnarray*}\nonumber
2\pi N^{j}(R)&=&\int_{0}^{R}\frac{d\arg\Pi_{R}^{j}(x-iK)}{dx}dx-\int_{0}^{R}\frac{\arg \Pi_{R}^{j}(x+iK)}{dx}dx
\\&&+\Im\big\{\int_{-K}^{K} \frac{  [\Pi_{R}^{j}]'(R+iy)}{\Pi_{R}^{j}(R+iy)}-\frac{[\Pi_{R}^{j}]'(iy)}{\Pi_{R}^{j}(iy)} dy\big\}\\\nonumber
&=&\int_{0}^{R}\frac{d\ln | \Pi_{R}^{j}(x-iK)|}{dK}dx+\int_{0}^{R}\frac{d\ln |\Pi_{R}^{j}(x+iK)|}{dK}dx+C_{R}^{j},\,R\gg1,\,K\gg0,
\end{eqnarray*}
in which we set 
\begin{equation}
C_{R}^{j}=\Im\big\{\int_{-K}^{K}  \frac{ [\Pi_{R}^{j}]'(R+iy)}{\Pi_{R}^{j}(R+iy)}-\frac{[\Pi_{R}^{j}]'(iy)}{\Pi_{R}^{j}(iy)}  dy\big\},
\end{equation}
where $C_{R}^{j}$ is bounded by the construction of $\partial 
S_{R}$ with finite width $2K$. Here we note that $C_{R}^{j}\rightarrow0$ as $K\downarrow0$ by Lebesgue Theorem.
\par
Using change of variable,
\begin{eqnarray*}
2\pi N^{j}(R)&=&-i\int_{x=0}^{x=R}\frac{d\ln | \Pi_{R}^{j}(x-iK)|}{d(x-iK)}d(x-iK)+i\int_{x=0}^{x=R}\frac{d\ln |\Pi_{R}^{j}(x+iK)|}{d(x+iK)}d(x+iK)+C_{R}^{j}\\
&=&-i[\ln | \Pi_{R}^{j}(R-iK)|- \ln| \Pi_{R}^{j}(-iK)|]+i[ \ln|\Pi_{R}^{j}(R+iK)|-\ln|\Pi_{R}^{j}(iK)|]+C_{R}^{j}.
\end{eqnarray*}
Hence,
\begin{eqnarray*}
2\pi iN^{j}(R)=\ln\frac{|\Pi_{R}^{j}(R-iK)|}{|\Pi_{R}^{j}(-iK)|}-\ln\frac{|\Pi_{R}^{j}(R+iK)|}{|\Pi_{R}^{j}(iK)|}+iC_{R}^{j}.
\end{eqnarray*}
In this case, we consider the subtraction
\begin{eqnarray*}
2\pi i\{ N^{1}(R)-N^{2}(R)\}&=&\ln|\frac{\Pi_{R}^{1}(R-iK)}{\Pi_{R}^{2}(R-iK)}|-\ln|\frac{\Pi_{R}^{1}(R+iK)}{\Pi_{R}^{2}(R+iK)}|\\
&&-\ln|\frac{\Pi_{R}^{1}(-iK)}{\Pi_{R}^{2}(-iK)}|+\ln|\frac{\Pi_{R}^{2}(iK)}{\Pi_{R}^{1}(iK)}|\\
&&+iC_{R}^{1}-iC_{R}^{2}.
\end{eqnarray*}
Let us set $$C_{R}:=-\ln|\frac{\Pi_{R}^{1}(-iK)}{\Pi_{R}^{2}(-iK)}|+\ln|\frac{\Pi_{R}^{2}(iK)}{\Pi_{R}^{1}(iK)}|+iC_{R}^{1}-iC_{R}^{2},$$
in which the $C_{R}$ is a constant depending on $R$ and $K$.
Hence,
\begin{eqnarray}\label{3.10}
2\pi i\{ N^{1}(R)-N^{2}(R)\}=\ln\{\frac{|\Pi_{R}^{1}(R-iK)|}{|\Pi_{R}^{2}(R-iK)|}\frac{|\Pi_{R}^{2}(R+iK)|}{|\Pi_{R}^{1}(R+iK)|}\}+C_{R}.
\end{eqnarray}
Now we use Lemma \ref{L29} to set $K\downarrow0$ in~(\ref{3.10}) to deduce $C_{R}\rightarrow0$, and
\begin{eqnarray}
2\pi i\{ N^{1}(R)-N^{2}(R)\}&\rightarrow&\ln\{\frac{|\Pi_{R}^{1}(R+i0^{-})|}{|\Pi_{R}^{1}(R+i0^{+})|}\frac{|\Pi_{R}^{2}(R+i0^{+})|}{|\Pi_{R}^{2}(R+i0^{-})|}\}\\
&=&\ln\{\frac{|\Pi_{R}^{1}(R+i0^{-})|}{|\Pi_{R}^{1}(R+i0^{+})|}\}-\ln\{\frac{|\Pi_{R}^{2}(R+i0^{-})|}{|\Pi_{R}^{2}(R+i0^{+})|}\}.\label{3.12}
\end{eqnarray}
On the other hand, we deduce
\begin{eqnarray}\nonumber
&&\hspace{6pt}2\pi i \{N^{1}(R)-N^{2}(R)\}\\\nonumber
&=&\int_{\partial S_{R}}[\ln\Pi^{1}_{R}(z)]'-[\ln\Pi^{2}_{R}(z)]'dz\\\nonumber
&=&\sum_{z^{1}_{n}\in S_{R},\,z^{2}_{n}\in S_{R}}\int_{\partial S_{R}}\frac{z_{n}^{1}-z_{n}^{2}}{(z-z_{n}^{1})(z-z_{n}^{2}))}dz\\\nonumber
&=&\sum_{\{{z^{1}_{n}\in S_{R}},{z^{1}_{n}\in S_{R}}\}}(z_{n}^{1}-z_{n}^{2}) \int_{\partial S_{R}}\frac{1}{(z-z_{n}^{1})(z-z_{n}^{2})}dz,
\end{eqnarray}
in which $\{z_{n}^{1}\}$ and $\{z_{n}^{2}\}$ are sufficiently close by assumption, that is , $0<|z_{n}^{1}-z_{n}^{2}|\ll1$. 
Using Proposition \ref{D},  $|N^{1}(R)-N^{2}(R)|$ is is uniformly bounded with respect to $R$.
\par
Furthermore, we let $S_{R_{n}}$ be the suitable rectangle in~(\ref{2.11}) that merely contains zero $z_{n}^{1}$ and $z_{n}^{2}$ for $n\gg1$, and deduce 
\begin{eqnarray*}
&&|\sum_{\{{z^{1}_{n}\in S_{R}},{z^{1}_{n}\in S_{R}}\}}(z_{n}^{1}-z_{n}^{2})\int_{\partial S_{R}}\frac{1}{(z-z_{n}^{1})(z-z_{n}^{2})}dz|\\
&=&|\sum_{\{{z^{1}_{n}\in S_{R}},{z^{1}_{n}\in S_{R}}\}}(z_{n}^{1}-z_{n}^{2})\int_{\partial S_{R_{n}}}\frac{1}{(z-z_{n}^{1})(z-z_{n}^{2})}dz|\\
&\leq&\sum_{\{{z^{1}_{n}\in S_{R}},{z^{1}_{n}\in S_{R}}\}}|z_{n}^{1}-z_{n}^{2}||\int_{\partial S_{R_{n}}}\frac{1}{(z-z_{n}^{1})(z-z_{n}^{2})}dz|\\
&\leq&\delta\sum_{\{{z^{1}_{n}\in S_{R}},{z^{1}_{n}\in S_{R}}\}}|\int_{\partial S_{R_{n}}}\frac{1}{(z-z_{n}^{1})(z-z_{n}^{2})}dz|,
\end{eqnarray*}
Without loss of generality, we assume $S_{R_{n}}$ is a square of width $2$, and $|\partial S_{R_{n}}|=8$. From Lemma \ref{L29}, we have 
$$\int_{\partial S_{R_{n}}}\frac{z_{n}^{1}-z_{n}^{2}}{(z-z_{n}^{1})(z-z_{n}^{2})}dz=0,\,\delta\ll1,\,n\gg1.$$
For the finite summation, we obtain
$$\sum_{\{{z^{1}_{n}\in S_{R}},{z^{1}_{n}\in S_{R}}\}}|\int_{\partial S_{R_{n}}}\frac{z_{n}^{1}-z_{n}^{2}}{(z-z_{n}^{1})(z-z_{n}^{2})}dz|<\delta M,$$
for some $M>0$.
In this case,
\begin{eqnarray}\label{3.13}
2\pi \{N^{1}(R)-N^{2}(R)\}<\delta M.
\end{eqnarray}
Since $\{N^{1}(R)-N^{2}(R)\}$ is an integer for any $R\gg1$, we conclude from~(\ref{3.13}) that $N^{1}(R)=N^{2}(R)$ for fixed $R\gg1$ as $\delta\rightarrow0$.
Using~(\ref{3.12}), we deduce that 
\begin{equation} 
\{z_{n}^{1}\}\equiv\{z_{n}^{2}\} \mbox{ for } z_{n}^{1},z_{n}^{2}\in S_{R}.
\end{equation}

\par
Now we use the Cartwright's Theorem \ref{Cartwright}. To deduce that: for any $\epsilon>0$ and $z$ on real axis, there exists an $R_{0}\gg1$ such that
\begin{equation}
\big|\hat{V}^{1}_{R}(z)\overline{\hat{V}^{1}}_{R}(z)-\hat{V}^{2}_{R}(z)\overline{\hat{V}^{2}}_{R}(z)\big|<\epsilon,\mbox{ whenever }R\geq R_{0}.
\end{equation}
That is 
\begin{equation}
\big||\hat{V}^{1}_{R}(z)|^{2}-|\hat{V}^{2}_{R}(z)|^{2}\big|<\epsilon,\mbox{ whenever }R\geq R_{0}.
\end{equation}

\end{proof}

\end{document}